\documentclass[12pt,reqno]{amsart}

\usepackage{amsaddr}

\usepackage{a4wide}

\usepackage[swedish,english]{babel}
\usepackage[T1]{fontenc}
\usepackage[latin1]{inputenc}

\usepackage{amsmath}
\usepackage{amssymb,bbm}
\usepackage{amsthm}
\usepackage{enumerate}
\usepackage{hyperref}

\newcommand{\Z}{\mathbb{Z}}
\newcommand{\Q}{\mathbb{Q}}

\newcommand{\F}{\mathbb{F}}

\DeclareMathOperator{\degree}{deg}
\DeclareMathOperator{\Char}{char}

\theoremstyle{plain}
\newtheorem{thm}{Theorem}[section]
\newtheorem{lem}[thm]{Lemma}
\newtheorem{prop}[thm]{Proposition}

\theoremstyle{definition}
\newtheorem{defn}[thm]{Definition}

\newtheorem{exmp}[thm]{Example}

\newtheorem{rem}[thm]{Remark}

\begin{document}

\title{Bimodules in differential polynomial rings}

\author{Johan \"{O}inert}
\address[Johan \"{O}inert]{Department of Mathematics and Natural Sciences, 
Blekinge Institute of Technology, 
SE-371 79 Karlskrona, Sweden \\
and \\
Department of Engineering, 
University of Sk\"{o}vde, 
SE-541 28 Sk\"{o}vde, Sweden}
\email{johan.oinert@bth.se}

\date{\today}


\begin{abstract}
We study the $R$-sub-bimodule structure of differential polynomial rings $R[x;\delta]$
by introducing the notion of \emph{strong simplicity}, requiring each nonzero $R$-sub-bimodule of $R[x;\delta]$ to be either $R[x;\delta]$ or the truncation $\sum_{i=0}^n R x^i$ for some $n \in \Z_{\geq 0}$.
Our main result gives a complete characterization: $R[x;\delta]$ is strongly simple if and only if $R$ is simple, $\Char(R)=0$, and the derivation $\delta$ is outer. 
We provide examples illustrating both when strong simplicity fails and when it holds.
\end{abstract}

\maketitle

\pagestyle{headings}


\section{Introduction}

Given a discrete group $G$ acting on a von Neumann algebra $M$ by $*$-automorphisms, one may
define a crossed product von Neumann algebra $M \rtimes G$ containing $M$ as a subalgebra.
Under the assumption that $M$ is a \emph{factor}
and that the action of $G$ on $M$ is \emph{outer}, 
Cameron and Smith \cite{CameronSmith2015}
showed that there is a bijection between subsets of $G$ and (closed) $M$-sub-bimodules of $M \rtimes G$ (see \cite[Thm.~4.4]{CameronSmith2015}).
In \cite{Oinert2017}, \"{O}inert proved algebraic analogues of Cameron and Smith's result for 
$G$-crossed products and further generalized it to the setting of group-graded rings.

The aim of the present article is to obtain analogous results in the context of differential polynomial rings.

Throughout this article, let $R$ be a unital ring that is not the zero ring.
Furthermore, let $\delta : R \to R$ be a \emph{derivation}, i.e. $\delta$ is additive and satisfies Leibniz's rule;
$\delta(ab)=\delta(a)b+a\delta(b)$, for all $a,b\in R$.
The corresponding \emph{differential polynomial ring} $R[x;\delta]$ is the free left $R$-module with $\{x^i\}_{i\in \Z_{\geq 0}}$ as its basis, and with 
multiplication defined by $x^i x^j = x^{i+j}$, for non-negative integers $i,j$, and the rule
\begin{equation}
	xr = rx + \delta(r)
\end{equation}
for every $r\in R$. 
One can verify that $R[x;\delta]$ is an associative and unital ring (see e.g. \cite{Nystedt2013} or \cite[Chap.~2]{GoodearlWarfield}).

Note that
\begin{displaymath}
	R \subseteq R+Rx \subseteq R+Rx+Rx^2 \subseteq \ldots \subseteq \sum_{i=0}^k Rx^i \subseteq \ldots \subseteq R[x;\delta]
\end{displaymath}
is an infinite chain of $R$-sub-bimodules of $R[x;\delta]$. It is natural to ask when every nonzero $R$-sub-bimodule of $R[x;\delta]$ appears in the above chain.

\begin{defn}\label{DefBM}
A differential polynomial ring $R[x;\delta]$ is said to be \emph{strongly simple}
if each nonzero $R$-sub-bimodule of $R[x;\delta]$
is equal to $R[x;\delta]$ or $\sum_{i=0}^n R x^i$ for some $n \in \Z_{\geq 0}$.
\end{defn}

Recall that a derivation $\delta : R \to R$ is said to be \emph{inner} if there is some $a\in R$ such that $\delta(r)=ar-ra$ for every $r\in R$.
A derivation that is not inner is called \emph{outer}.
The center of $R$ will be denoted by $Z(R)$,
and the centralizer of $R$ in an overring $S$
is the set $C_S(R):=\{s \in S \mid sr=rs \text{ for all } r\in R\}$.
We establish the following characterization of strong simplicity in differential polynomial rings.

\begin{thm}\label{thm:Main}
Let $S:=R[x;\delta]$ be a differential polynomial ring.
The following four assertions are equivalent:
\begin{enumerate}[{\rm (i)}]
	\item $R[x;\delta]$ is strongly simple.
	\item $R$ is a simple ring with $\Char(R)=0$ and $S$ is simple.
	\item $R$ is a simple ring with $\Char(R)=0$ and $\delta$ is outer.
	\item $R$ is a simple ring with $\Char(R)=0$ and $C_S(R)=Z(R)$.
\end{enumerate}
\end{thm}

It is worth noting that the centralizer condition appearing in (iv) above is very similar to the centralizer condition appearing in the corresponding characterization in the setting of strongly group-graded rings (see \cite[Thm.~2]{Oinert2017}).

This article is organized as follows.
%
In Section~\ref{Sec:Necessary}, we establish necessary conditions for strong simplicity (see Proposition~\ref{Prop:NecessaryConditions}).
We also present two well-known examples of differential polynomial rings that are not strongly simple (see Example~\ref{ex:PolynomialRing} and Example~\ref{ex:WeylAlgebra}). 
In Section~\ref{Sec:Sufficient}, we establish sufficient conditions for strong simplicity (see Proposition~\ref{prop:Sufficient}).
We also present two examples of differential polynomial rings that are strongly simple (see Example~\ref{ex:RationalFunctions} and Example~\ref{ex:RationalFunctionsMatrix}).
In Section~\ref{Sec:MainResult}, we prove Theorem~\ref{thm:Main}.

\section{Necessary conditions}\label{Sec:Necessary}

In this section, we establish necessary conditions for strong simplicity.

Henceforth, if $a \in R[x;\delta]$, then we write
$a= \sum_{k\geq 0} a_k x^k$
for its unique expression with coefficients $a_k \in R$, 
all but finitely many of which are zero.
Thus $a_k$ is uniquely determined by $a$ and denotes the coefficient of $a$ in front of $x^k$.
If $a\in R[x;\delta]$ is nonzero, then we define $\degree(a):=\max\{k \in \Z_{\geq 0} \mid a_k \neq 0 \}$.

\begin{lem}\label{lem:SimpleCentralizerOuter}
Let $S:=R[x;\delta]$ be a differential polynomial ring.
Suppose that at least one of the following assertions holds:
\begin{enumerate}[{\rm (i)}]
	\item $S$ is simple.
	\item $C_S(R)=Z(R)$.
\end{enumerate}
Then $\delta$ is outer.
\end{lem}

\begin{proof}
If $\delta$ is inner, then
there is some $a\in R$ such that $\delta(r)=ar-ra$, for every $r\in R$.
We also have $xr=rx+\delta(r)$, i.e. $\delta(r)=xr-rx$.
Hence, 
$ar-ra=xr-rx \Leftrightarrow (a-x)r=r(a-x)$, for every $r\in R$.

(i) Suppose that $S$ is simple.
Seeking a contradiction, suppose that $\delta$ is inner.
By the above calculation, $a-x \in C_S(R)$.
Using that $\delta(a)=a^2-a^2=0$, we get that
$x(a-x)=xa-x^2=ax+\delta(a)-x^2=ax-x^2=(a-x)x$.
Thus, $a-x \in Z(R[x;\delta])$.
The ideal of $S$ generated by $a-x$ is nonzero and it must be proper, because otherwise $a-x$ would be invertible 
which is not possible since $\degree((a-x)b)>0$ for any nonzero $b\in R[x;\delta]$. 
This is a contradiction. Hence, $\delta$ is outer.

(ii) Suppose that $C_S(R)=Z(R)$.
Seeking a contradiction, suppose that $\delta$ is inner.
Note that $a-x \in C_S(R)=Z(R) \subseteq R$.
This is a contradiction. Hence, $\delta$ is outer.
\end{proof}

\begin{rem}
The first half of the above lemma has already appeared in \cite[Lem.~4.1.3]{Jordan1975}.
\end{rem}

Recall that in a differential polynomial ring the identity
\begin{equation}\label{eq:CommutationRelation}
	x^n r = \sum_{i=0}^n \binom{n}{i} \delta^{n-i}(r) x^i
\end{equation}
holds for any $r\in R$ and $n\in \Z_{>0}$.

\begin{lem}\label{lem:CharZeroNecessary}
Let $S:=R[x;\delta]$ be a differential polynomial ring which is strongly simple.
Suppose that $R$ is simple.
Then $\Char(R)=0$.
\end{lem}

\begin{proof}
By assumption, $\Char(R)\neq 1$.
Seeking a contradiction, suppose that $\Char(R)\neq 0$.
By simplicity of $R$, we must have $\Char(R)=p$ for some prime number $p$.
Recall that $p$ is a factor in $\binom{p}{i}$ for $i\in \{1,\ldots,p-1\}$. 
Using \eqref{eq:CommutationRelation}, we see that $x^p r=rx^p+\delta^p(r)$ for every $r\in R$.
Hence, $M:=R + Rx^p$ is an $R$-sub-bimodule of $R[x;\delta]$.
Since $p>1$, this contradicts strong simplicity. Thus, $\Char(R)=0$.
\end{proof}

\begin{prop}\label{Prop:NecessaryConditions}
Let $S:=R[x;\delta]$ be a differential polynomial ring which is strongly simple.
The following five assertions hold:
\begin{enumerate}[{\rm (a)}]
	\item $S$ is a simple ring.
	\item $R$ is a simple ring.
	\item $C_S(R)=Z(R)$.
	\item $\delta$ is outer.
	\item $\Char(R)=0$.
\end{enumerate}
\end{prop}

\begin{proof}
(a)
Let $I$ be a nonzero ideal of $S$.
Then $I$ is also a nonzero $R$-bimodule.
Note that there cannot be any $n \in \Z_{\geq 0}$ such that $I=\sum_{i=0}^n R x^i$, because then $Ix \nsubseteq I$ would contradict the fact that $I$ is an ideal of $S$.
Therefore, by assumption, we must have $I=S$.

(b)
Each nonzero ideal $J$ of $R$ is a nonzero $R$-bimodule. Hence, $J=R x^0=R$.

(c)
Clearly, $Z(R) \subseteq C_S(R)$.
We now show that $C_S(R) \subseteq Z(R)$.
Take $a= \sum_{i=0}^n a_i x^i \in C_S(R)$ with $a_n \neq 0$.
For any $r\in R$ we have $ar=ra$, and hence $(ar)_n=(ra)_n$, that is $a_n r = r a_n$.
By (b), $R$ is simple and hence $a_n \in Z(R)$ is invertible.
Put $T_n := \sum_{i=0}^n R x^i$.
Since $a \in C_S(R)$, the set $Ra$ is a nonzero $R$-sub-bimodule of $S$.
Moreover, $Ra \subseteq T_n$, and $Ra$ contains $a$, which has degree $n$.
Hence, by strong simplicity, $Ra = T_n$.

Seeking a contradiction, suppose that $n>0$.
Then $1_R \in T_n = Ra$, so $1_R=ra$ for some $r\in R$.
Hence, $(ra)_n=0$, that is $r a_n = 0$.
Using that $a_n$ is invertible we get that $r=0$, contradicting $ra=1_R$.
Therefore, $n=0$. Hence, $a \in R \cap C_S(R) = Z(R)$.
	
(d)
This follows from (a) or (c), and Lemma~\ref{lem:SimpleCentralizerOuter}.

(e)
This follows from (b) and Lemma~\ref{lem:CharZeroNecessary}.
\end{proof}

\begin{exmp}\label{ex:PolynomialRing}
The usual polynomial ring $R[x]$, corresponding to $\delta=0$, is not simple and can therefore never be strongly simple.
\end{exmp}

\begin{exmp}\label{ex:WeylAlgebra}
Let $\F$ be a field with $\Char(\F)=0$ and consider the first Weyl algebra $A_1(\F) := \F\langle x,y \rangle / (xy-yx-1)$.
Then $A_1(\F) \cong \F[y][x;\delta]$ where $\delta := \frac{\partial}{\partial y}$.
It is well known that $A_1(\F)$ is a simple ring. But $\F[y]$ is not simple and hence, by Proposition~\ref{Prop:NecessaryConditions}, $\F[y][x;\delta]$ cannot be strongly simple.
\end{exmp}

\section{Sufficient conditions}\label{Sec:Sufficient}

In this section, we establish sufficient conditions for strong simplicity.

\begin{lem}\label{lem:HighestDegreesCommutation}
Let $S:=R[x;\delta]$ be a differential polynomial ring.
Suppose that $c = \sum_{k=0}^n c_k x^k \in R[x;\delta]$ is a nonzero element of degree $n>0$.
For any $r\in R$, the following equalities hold:
\begin{enumerate}[{\rm (a)}]
	\item $(rc-cr)_n = rc_n - c_n r$
	\item $(rc-cr)_{n-1} = r c_{n-1} - n c_n \delta(r) - c_{n-1}r$
\end{enumerate}
\end{lem}

\begin{proof}
Take any $r\in R$.
Using \eqref{eq:CommutationRelation} we get
\begin{align*}
	rc-cr&= \sum_{k=0}^{n} r c_k x^k - \left(\sum_{k=0}^{n} c_k x^k  \right) r 
	= \sum_{k=0}^{n} r c_k x^k - \left(\sum_{k=0}^{n} \sum_{i=0}^k \binom{k}{i} c_k \delta^{k-i}(r) x^i \right) \\
	&= r c_n x^n + r c_{n-1} x^{n-1} - c_n rx^n - n c_n \delta(r) x^{n-1} - c_{n-1}r x^{n-1} + [\text{lower degree terms}] \\
	&= (rc_n - c_n r) x^n + (r c_{n-1} - n c_n \delta(r) - c_{n-1}r) x^{n-1} + [\text{lower degree terms}]
\end{align*}
from which the equalities (a) and (b) follow directly.
\end{proof}

Recall that a nonzero element in $R[x;\delta]$ is said to be \emph{monic} if its highest degree coefficient is equal to $1_R$.

\begin{lem}\label{lem:MonicRSubmodule}
Let $S:=R[x;\delta]$ be a differential polynomial ring
for which $R$ is a simple ring, and such that $C_S(R)=Z(R)$.
Suppose that $M$ is a nonzero $R$-sub-bimodule of $S$.
The following two assertions hold:
\begin{enumerate}[{\rm (a)}]
	\item If $M$ contains an element of degree $n$, then
				$M$ contains a monic element of degree $n$.
	\item $R$ is an $R$-sub-bimodule of $M$.
\end{enumerate}
\end{lem}

\begin{proof}
(a)
Let $a=\sum_{i=0}^n a_i x^i$ be an element of degree $n$, i.e. $a_n\neq 0$.
By simplicity of $R$, there are $\alpha_1,\ldots,\alpha_k,\beta_1,\ldots,\beta_k \in R$
such that
$\sum_{j=1}^k \alpha_j a_n \beta_j = 1_R$.

Consider the element
\begin{displaymath}
	a' := \sum_{j=1}^k \alpha_j a \beta_j \in M.
\end{displaymath}
Note that $(a')_n = \sum_{j=1}^k \alpha_j a_n \beta_j = 1_R$. That is, $a'$ is monic.

(b)
Let $a\in M$ be a nonzero element of smallest possible degree. 
Seeking a contradiction, suppose that $m:=\degree(a)>0$. By (a) we may assume, without loss of generality, that $a$ is monic.
Take $r\in R$.
By Lemma~\ref{lem:HighestDegreesCommutation}(a), and the fact that $a$ is monic, we get that 
$a':=ra-ar \in M$ is either zero or $\degree(a')<m$.
By the minimality of $m$ we conclude that $a'=0$, i.e. $ra=ar$.
Using that $r$ was arbitrarily chosen, this shows that $a\in C_S(R)=Z(R) \subseteq R$, contradicting $\degree(a)=m>0$. 
Thus, $m=0$.
By simplicity of $R$ we conclude that $R \subseteq M$.
\end{proof}

\begin{rem}\label{rem:n-invertible}
If $R$ is a simple unital ring with $\Char(R)=0$, then for every $n\in \Z_{>0}$, the element $n 1_R$ is invertible in $R$.
Indeed, by simplicity, $n 1_R R = R$ from which it follows that $n 1_R$ is invertible.
\end{rem}

\begin{lem}\label{lem:Degrees0tok}
Let $S:=R[x;\delta]$ be a differential polynomial ring
for which $R$ is a simple ring with $\Char(R)=0$, and such that $C_S(R)=Z(R)$. 
Suppose that $M$ is a nonzero $R$-sub-bimodule of $S$.
If $M$ contains an element of degree $n$, then
$M$ contains a monic element of degree $k$ for each $k\in \{0,\ldots,n\}$.
\end{lem}

\begin{proof}
Suppose that $M$ contains an element of degree $n$.
If $n=0$, then the claim follows from Lemma~\ref{lem:MonicRSubmodule}(b).
Now, assume that $n>0$. 
By Lemma~\ref{lem:MonicRSubmodule}(a),
$M$ contains a monic element of degree $n$.

Seeking a contradiction, suppose that the assertion fails.
Let $m$ be the largest integer
in $\{0,\ldots,n\}$ such that $M$ contains no monic element of degree $m$.
Note that $0<m<n$.
By the maximality of $m$, $M$ contains a monic element
$a=x^{m+1}+\sum_{i=0}^m a_ix^i$.

For $r\in R$, consider $a':=ra-ar \in M$. 
By Lemma~\ref{lem:HighestDegreesCommutation}(a), we have $(a')_{m+1}=r 1_R - 1_Rr=0$.
We claim that, for some $r\in R$, we must have $(a')_m \neq 0$.
Otherwise, Lemma~\ref{lem:HighestDegreesCommutation}(b) yields
$ra_m-(m+1)1_R \delta(r)-a_mr=0$
for every $r\in R$. Using that $\Char(R)=0$, Remark~\ref{rem:n-invertible} gives that
$(m+1)1_R$ is invertible. Hence,
$\delta(r) = (-(m+1)^{-1}a_m)r - r(-(m+1)^{-1}a_m)$
for every $r\in R$. 
Thus $\delta$ is inner, which is a contradiction, by Lemma~\ref{lem:SimpleCentralizerOuter}, because $C_S(R)=Z(R)$.

Therefore there is some $r\in R$ such that $a'=ra-ar$ has degree $m$. By
Lemma~\ref{lem:MonicRSubmodule}(a), $M$ contains a monic element of degree $m$, contradicting the choice
of $m$. Hence, $M$ contains a monic element of degree $k$
for every $k\in\{0,\ldots,n\}$.
\end{proof}

\begin{lem}\label{lem:SubBiModuleInductive}
Let $S:=R[x;\delta]$ be a differential polynomial ring
for which $R$ is a simple ring with $\Char(R)=0$, and such that $C_S(R)=Z(R)$.
Suppose that $M$ is a nonzero $R$-sub-bimodule of $S$ containing an element of degree $n$.
Then $\sum_{i=0}^{n} Rx^i \subseteq M$.
\end{lem}

\begin{proof}
By Lemma~\ref{lem:MonicRSubmodule}(b), we have $R \subseteq M$.
Suppose that $n>0$.
We claim that $R x^k \subseteq M$ for every $k \in \{0,\ldots,n\}$.
If we assume that the claim holds, then $\sum_{i=0}^n Rx^i \subseteq M$.
Now we prove the claim by induction.
The base case $k=0$ is clear.
By Lemma~\ref{lem:Degrees0tok}, $M$ contains monic elements of degrees $0, \ldots, n$.
Take $k \in \{1,\ldots,n\}$ and suppose that $\sum_{i=0}^{k-1} Rx^i \subseteq M$.
Let $b = x^k + \sum_{i=0}^{k-1} b_ix^i \in M$ be monic of degree $k$.
Then
$x^k = b - \sum_{i=0}^{k-1} b_ix^i \in M$.
Hence, $Rx^k \subseteq M$.
\end{proof}

\begin{prop}\label{prop:Sufficient}
Let $S:=R[x;\delta]$ be a differential polynomial ring 
for which $R$ is a simple ring with $\Char(R)=0$, and such that $C_S(R)=Z(R)$.
Suppose that $M$ is a nonzero $R$-sub-bimodule of $S$.
Then either
$M=R[x;\delta]$
or
$M=\sum_{i=0}^n R x^i$ for some $n\in \Z_{\geq 0}$.
That is, $R[x;\delta]$ is strongly simple.
\end{prop}

\begin{proof}
Write $T_k:=\sum_{i=0}^k Rx^i$ for $k\in \Z_{\geq 0}$. 
If the degrees of elements of $M$ are unbounded, then for every $k\geq 0$, $M$ contains an element of degree at least $k$. Hence, by Lemma~\ref{lem:SubBiModuleInductive}, $T_k \subseteq M$. Thus, $M=S$.
Otherwise, let
$m:=\max\{\deg(a)\mid a\in M \setminus \{0\}\}$.
Then $M\subseteq T_m$.
Since $M$ contains an element of degree $m$, Lemma~\ref{lem:SubBiModuleInductive}
yields $T_m \subseteq M$. Thus, $M=T_m$.
This shows that $R[x;\delta]$ is strongly simple.
\end{proof}

\begin{lem}\label{lem:CentralizerOuter}
Let $S:=R[x;\delta]$ be a differential polynomial ring.
Suppose that $R$ is simple with $\Char(R)=0$, and that $\delta$ is outer.
Then $C_S(R)=Z(R)$.
\end{lem}

\begin{proof}
We prove the contrapositive statement.
Suppose that
$C_S(R)\neq Z(R)$ and
choose some $c = \sum_{i=0}^n c_i x^i \in C_S(R)$ of degree $n>0$.

For every $r\in R$, we have $rc=cr$, and hence by Lemma~\ref{lem:HighestDegreesCommutation}, it follows that
\begin{itemize}
	\item $rc_n = c_n r$, and 
	\item $rc_{n-1} = n c_n \delta(r) + c_{n-1}r$.
\end{itemize}
Thus, $c_n \in Z(R)$, and $c_n R$ is a nonzero ideal of $R$.
Hence, by simplicity of $R$, we have $c_n R = R$. In particular, $c_n$ is invertible in $R$.
From the above and Remark~\ref{rem:n-invertible}, we get that
$\delta(r) = (-c_{n-1} (nc_n)^{-1}) r - r (-c_{n-1} (n c_n)^{-1})$ for every $r\in R$.
This shows that $\delta$ is inner.
\end{proof}

\begin{exmp}\label{ex:RationalFunctions}
Let $R:=\Q(t)$ be the field of rational functions in one indeterminate $t$ over $\Q$, the field of rational numbers.
Clearly, $R$ is simple and $\Char(R)=0$.
Define $\delta : R \to R$ by $\delta(f):=\frac{\partial f}{\partial t}$.
It is easy to verify that $\delta$ is a nonzero derivation on the commutative ring $R$.
Hence, $\delta$ is outer.
By Lemma~\ref{lem:CentralizerOuter} and Proposition~\ref{prop:Sufficient}, 
the differential polynomial ring $R[x;\delta]$ is strongly simple.
\end{exmp}

\begin{exmp}\label{ex:RationalFunctionsMatrix}
We now construct an example where $R$ is a non-commutative ring that is not a division ring.
Let $\delta : \Q(t) \to \Q(t)$ be the derivation from Example~\ref{ex:RationalFunctions}.
Put $R:=M_2(\Q(t))$.
The map $\widetilde{\delta} : R \to R$ defined by $[\widetilde{\delta}(f)]_{i,j} := \delta(f_{i,j})$
is a derivation. 
Clearly, $R$ is simple and $\Char(R)=0$.
Let $I_2$ denote the identity matrix in $R$.
Note that $\widetilde{\delta}(tI_2)=I_2$.
Seeking a contradiction, suppose that $\widetilde{\delta}$ is inner.
Then $\widetilde{\delta}(Z(R))=0$.
This contradicts the fact that $tI_2 \in Z(R)$.
We conclude that $\widetilde{\delta}$ is outer.
Hence, by Lemma~\ref{lem:CentralizerOuter} and Proposition~\ref{prop:Sufficient}, 
the differential polynomial ring $R[x;\widetilde{\delta}]$ is strongly simple.
\end{exmp}

\section{Proof of the main result}\label{Sec:MainResult}

In this section, we combine the preceding results and observations to prove the main result.

\begin{proof}[Proof of Theorem~\ref{thm:Main}]
(i)$\Rightarrow$(ii)
This follows from Proposition~\ref{Prop:NecessaryConditions}.

(ii)$\Rightarrow$(iii)
This follows from Lemma~\ref{lem:SimpleCentralizerOuter}.

(iii)$\Rightarrow$(iv)
This follows from Lemma~\ref{lem:CentralizerOuter}.

(iv)$\Rightarrow$(i)
This follows from Proposition~\ref{prop:Sufficient}.
\end{proof}

\begin{rem}
By \cite[Thm.~4.1.4]{Jordan1975}, 
the conditions appearing in Theorem~\ref{thm:Main}(iii)
imply that $R[x;\delta]$ is a simple ring.
By \cite[Thm.~4.15]{Oinert2013},
the conditions appearing in Theorem~\ref{thm:Main}(iv)
can also easily be shown to imply that $R[x;\delta]$ is a simple ring.
\end{rem}

\end{document}